\documentclass[lettersize,journal]{IEEEtran}
\usepackage{amsmath,amsfonts}
\usepackage{algorithmic}
\usepackage{algorithm}
\usepackage{array}
\usepackage[caption=false,font=normalsize,labelfont=sf,textfont=sf]{subfig}
\usepackage{textcomp}
\usepackage{stfloats}
\usepackage{url}
\usepackage{verbatim}
\usepackage{graphicx}
\usepackage{cite}
\usepackage{booktabs}
\usepackage{ellipsis}
\usepackage{mathtools}
\usepackage[european]{circuitikz}
\usepackage{tuda-pgfplots}

\usepackage{nicematrix}
\usepackage{blkarray}
\usepackage{amssymb}
\usepackage{amsthm}
\newtheorem{definition}{Definition}
\newtheorem{assumption}{Assumption}
\newtheorem{proposition}{Proposition}
\newtheorem{lemma}{Lemma}
\newtheorem{theorem}{Theorem}
\newtheorem{corollary}{Corollary}

\newtheorem{example}{Example}

\usepackage[hidelinks]{hyperref}
\usepackage[nameinlink, capitalise]{cleveref}
\crefname{equation}{}{}
\crefname{assumption}{Assumption}{Assumptions}
\crefname{figure}{Fig.}{Fig.}
\crefname{proof}{Proof}{Proofs}

\usepackage{accents}

\usepackage{notation}

\usetikzlibrary{decorations}
\usetikzlibrary{external}
\begin{document}
\title{A topological decoupling of modified nodal analysis\\ including controlled sources}

\author{Idoia Cortes Garcia, Peter F. Förster, Lennart Jansen, Wil Schilders, and Sebastian Schöps
}



\maketitle

\begin{abstract}
    We derive a topological decoupling of the equations of modified nodal analysis (MNA) to a semi-explicit index one differential-algebraic equation. The decoupling explicitly allows for controlled sources, which play a crucial role in engineering design workflows. Furthermore, the proof is constructive and provides a graph-based algorithmic framework for the computation of the decoupling, enabling its application to a variety of industry problems. These include the generation of consistent initial conditions, model order reduction, (scientific) machine learning, as well as speeding up conventional circuit simulation. In addition, the decoupling preserves the structure of MNA, i.e.\ the resulting systems remain sparse and key parts remain positive definite. We illustrate the decoupling using multiple examples, including some of the most common subcircuits containing controlled sources. Lastly, we also provide a first software implementation of the decoupling.
\end{abstract}


\section{Introduction}
\label{sec:int}
\IEEEPARstart{M}{odified} nodal analysis (MNA) \cite{ho1975} is one of the most common circuit formulations and forms the basis for SPICE-like circuit simulators such as LTspice \cite{ltspice}, Xyce \cite{xyce} and PSpice \cite{pspice}. Due to its popularity, MNA has been extensively studied not only from an engineering, but also from a mathematical perspective. In this context, it is important to note that MNA can be abstractly characterized as a differential-algebraic equation (DAE)
\begin{align*}
    \mb{f}(\mb{x}', \mb{x}, t) = \mb{0},
\end{align*}
where the Jacobian of $\mb{f}$ w.r.t.\ $\mb{x}'$ is singular. DAEs are inherently different from the more commonly encountered ordinary differential equations and can lead to significant numerical problems \cite{petzold1982}. As a consequence, so-called index concepts were developed to better understand and classify their behavior \cite{mehrmann2015}. Regarding MNA, one of the most well-known index results is due to Tischendorf \cite{tischendorf1999}. It provides a topological characterization of the tractability and differentiation indices of MNA, when only considering independent, but not controlled, sources. The extension to a topological result that also accounts for controlled sources was provided shortly after by Est{\'e}vez Schwarz and Tischendorf \cite{estevez2000}. For many purposes this result marks the state of the art, as it allows to verify whether a given circuit can be simulated using standard algorithms such as TR-BDF2 \cite{hosea1996}.

More recently, Jansen \cite{jansen2014} introduced the dissection index, which not only provides an equivalent topological index result for MNA, but also a topological decoupling of the system of MNA to a semi-explicit index one DAE\footnote{See \cref{def:seio} for a rigorous definition.}. Such a decoupling enables many applications, as it e.g.\ allows for efficient model order reduction or machine learning approaches along the lines of \cite{ali2014} and \cite{forster2023c}, while also facilitating the computation of consistent initial values fulfilling the hidden constraints of the DAE, compare \cite{estevez2021}. Moreover, it may provide an avenue towards speeding up conventional simulations, by leveraging the semi-explicit structure using appropriate time integrators, see e.g.\ \cite[Chapter 7]{jansen2014}.

A major obstacle regarding industry-level applications has so far been the issue that the topological decoupling is only available for circuits without controlled sources, as the latter are crucial for modeling common circuit elements such as transistors, switches or operational amplifiers. We now remove this obstacle, by generalizing the result from \cite{jansen2014} to one allowing for controlled sources underlying almost identical conditions as those introduced in \cite{estevez2000}. Furthermore, we also provide a first software implementation of the decoupling \cite{forster2026c}.


The structure of the paper is as follows. \Cref{sec:pre} collects prerequisites for the decoupling presented in \cref{sec:dec}, while \cref{sec:ex} provides a number of examples that include some of the most common controlled circuit elements. Lastly, \cref{sec:con} draws conclusions.

\section{Prerequisites}
\label{sec:pre}
We begin by listing a few definitions and assumptions that are required for the decoupling. First, we provide a rigorous definition for semi-explicit index one DAEs, see also \cite[Section 2.2]{brenan1995}, as this marks the goal of the decoupling.

\begin{definition}[Semi-explicit index one DAE]
    \label{def:seio}
    A DAE of the form
    \begin{subequations}
        \label{eq:seio}
        \begin{align}
            \dd{t} \mb{x} &= \mb{f}(\mb{x}, \mb{y}, t) \label{eq:seio.1}\\
            \mb{0} &= \mb{g}(\mb{x}, \mb{y}, t) \label{eq:seio.2}
        \end{align}
    \end{subequations}
    is called semi-explicit index one DAE if the Jacobian of $\mb{g}$ w.r.t.\ $\mb{y}$ is regular. (This implies that $\mb{g}$ is uniquely solvable for $\mb{y}$ by the implicit function theorem.) We further refer to $\mb{x}$ as the \emph{differential} variables and $\mb{y}$ as the \emph{algebraic} variables.
\end{definition}

Next, we recall the basic definition of an incidence matrix, as it represents the fundamental building block of MNA.

\begin{definition}[Incidence matrix]
    \label{def:im}
    The entries $a_{i,j}$ of an unreduced incidence matrix $\A \in \mathbb{R}^{n \times b}$, describing a circuit with $n$ nodes and $b$ branches, are defined as
    \begin{align*}
        a_{i,j} \coloneqq \begin{cases}
            +1, &\text{if branch $j$ leaves node $i$,}\\
            -1, &\text{if branch $j$ enters node $i$,}\\
            0, &\text{if branch $j$ is not incident to node $i$.}
        \end{cases}
    \end{align*}
    A reduced incidence matrix $\A'$ is obtained by eliminating one row of $\A$ (this corresponds to choosing a ground node). In the following, we mostly work with reduced incidence matrices, thus these should be understood as the default. We therefore write $\A$ also for reduced incidence matrices and explicitly remark when dealing with unreduced ones.
\end{definition}

The following definition introduces the key concept used for the decoupling, compare also \cite[Definition 4.5]{jansen2014}.

\begin{definition}[Basis matrices]
    \label{def:bf}
    Let $\AT \in \mathbb{R}^{b \times n}$, then we define four basis matrices\footnote{Note that the original definition is given in terms of basis functions, as it allows for the matrix entries to depend on the solution \cite[Definition 4.5]{jansen2014}. We only deal with constant basis functions however, hence the name change.} $\mb{P}, \Q, \V, \W$ such that
    \begin{align*}
        \im \Q = \ker \AT, \quad \im \W = \ker \A
    \end{align*}
    and the columns of $[\mb{P}, \Q]$ form a basis of $\mathbb{R}^n$, while the columns of $[\V, \W]$ form a basis of $\mathbb{R}^b$.
\end{definition}

We can now present the assumptions needed for the decoupling, beginning with the standard assumptions required for the problem to be well-posed, compare also \cite[Theorem 2.2]{estevez2000}.

\begin{assumption}[Well-posedness assumptions]
    \label{as:wp}
    ~\begin{enumerate}
        \item The circuit is connected, not shorted and contains no self-loops.
        \item The circuit contains no $\mr{V}$-loops made up of only voltage sources.
        \item The circuit contains no $\mr{I}$-cutsets made up of only current sources.
    \end{enumerate}
\end{assumption}

The latter two of these topological assumptions can be translated into the following algebraic ones, compare again \cite[Theorem 2.2]{estevez2000}. (Subscripts indicate which element types the incidence matrices belong to, e.g.\ $\mr{V}$ signifies voltage sources.)

\begin{proposition}
    \label{prop:wp}
    \Cref{as:wp} implies that
    \begin{enumerate}
        \item[2)] $\A[V]$ has full column rank and that
        \item[3)] $[\A[C], \A[L], \A[R], \A[V]]$ has full row rank.
    \end{enumerate}
\end{proposition}

In addition to \cref{as:wp}, we also assume the topology of the circuit to be fixed. In cases where the topology changes over time, e.g.\ due to switches, the corresponding switching events need to be detected and resolved during time integration, see e.g.\ \cite{tant2018}. The decoupling may then be computed independently (in advance) for each relevant topology.

The following common assumption, compare e.g.\ \cite[Theorem 4.1]{estevez2000}, ensures all passive elements behave passively.

\begin{assumption}[Positive definite element Jacobians]
    \label{as:ej}
    The Jacobians
    \begin{align*}
        \mb{C}(\mb{x}) &\coloneqq \pp[\qC(\mb{x})]{\mb{x}}\\
        \mb{L}(\mb{x}) &\coloneqq \pp[\phiL(\mb{x})]{\mb{x}}\\
        \mb{G}(\mb{x}, \mb{y}) &\coloneqq \pp[\gR(\mb{x}, \mb{y})]{\mb{x}}
    \end{align*}
    of the functions $\qC$, $\phiL$ and $\gR$, describing capacitors, inductors and resistors respectively, are all positive definite.
\end{assumption}

We explicitly note that the function describing resistors contains an additional non-standard argument, as this will become important later. Furthermore, we remark that \cref{as:ej} also allows for multiport elements, compare \cite{tischendorf1999} and \cite{estevez2000}. This implies that elements of the same type, e.g.\ two resistors, can control each others values, as long as the resulting Jacobian is still positive definite. In particular this includes transformers, which can be interpreted as multiport inductors.

In the following, we distinguish between \emph{independent} sources that only depend on time, indicated by a subscript $\mr{s}$, and \emph{controlled} sources which may also depend on solution variables, indicated by a subscript $\mr{c}$. We then write $\A[V] = [\A[V_s], \A[V_c]]$ and $\A[I] = [\A[I_s], \A[I_c]]$ for the corresponding (reduced) incidence matrices of voltage and current sources, respectively. We further split the controlled current sources into four types $\AIc[i]$, $1 \leq i \leq 4$, which will be used to distinguish their topological properties, and write $\A[I_c] = [\AIc[1], \dotsc, \AIc[4]]$. Each controlled current source belongs to exactly one of the four types. Using this notation, we can now state the assumptions on controlled sources.

\begin{assumption}[Assumptions on controlled sources]
    \label{as:cs}
    ~\begin{enumerate}
        \item Controlled voltage sources are not part of any $\mr{C}$-$\mr{V}$ loop. (A $\mr{C}$-$\mr{V}$ loop is a loop made up of only capacitors or\footnote{We use or in the logical sense, meaning inclusive or.} voltage sources.)
        \item Controlled current sources described by $\AIc[1]$ have a $\mr{V_s}$-only path between their terminals.
        \item Controlled current sources described by $\AIc[2]$ have a $\mr{C}$-$\mr{V_s}$-only path between their terminals.
        \item Controlled current sources described by $\AIc[3]$ have a $\mr{C}$-$\mr{V}$-only path between their terminals.
        \item Controlled current sources described by $\AIc[4]$ are not part of any $\mr{L}$-$\mr{I}$ cutset. (An $\mr{L}$-$\mr{I}$ cutset is a cutset made up of only inductors or current sources.)
    \end{enumerate}
\end{assumption}

In order to gain some intuition for these assumptions, we compare them to those of the topological index result about MNA \cite[Theorem 4.1]{estevez2000}. In doing so, we first note that \cref{as:cs}.1 and \cref{as:cs}.5 can be seen to ensure that controlled sources do not appear differentiated in time in any of the expressions resulting from the decoupling. Furthermore, the remaining assumptions ensure that the algebraic equations derived by the decoupling remain uniquely solvable for the algebraic variables, compare \cref{def:seio}, while allowing for as many controlled source contributions as possible. Similar to \cite{estevez2000}, we also remark that the listed conditions are sufficient, but not necessary, i.e.\ there might be circuits that violate one of the assumptions but still allow for a topological decoupling. Lastly, \cref{fig:cs_assumptions} shows an abstract illustration of the conditions, which we will come back to in \cref{pr:cs}.

\begin{figure*}
    \begin{center}
        \subfloat[$\mr{C}$-$\mr{V}$ loops]{\label{fig:cv_loops} \includegraphics[width=0.23\textwidth]{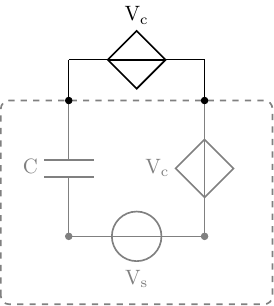}}
        \hspace{2em}
        \subfloat[$\mr{C}$-$\mr{V}$-only paths]{\label{fig:cv_only_paths} \includegraphics[width=0.23\textwidth]{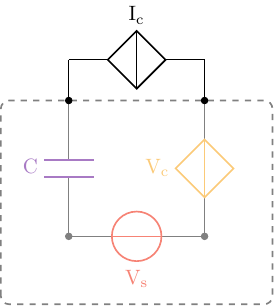}}
        \hspace{2em}
        \subfloat[$\mr{L}$-$\mr{I}$ cutsets]{\label{fig:li_cutsets} \includegraphics[width=0.23\textwidth]{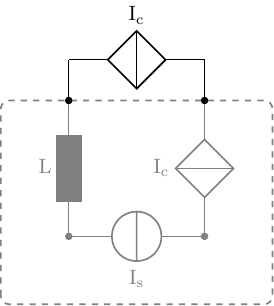}}
    \end{center}
    \caption{Illustrations of the conditions from \cref{as:cs}.}
    \label{fig:cs_assumptions}
\end{figure*}

As with \cref{as:wp}, we can again derive algebraic conditions from these topological assumptions. (The proof is provided later in \cref{pr:cs}.) Introducing a shorthand for the product of multiple basis matrices, e.g.\ $\Q[V_s C] \coloneqq \Q[V_s] \Q[C]$, the derived conditions read as follows. (The transpose of the shorthand should be read as, e.g., $\QT[V_s C] = \QT[C] \QT[V_s]$.)

\begin{proposition}
    \label{prop:cs}
    \Cref{as:cs} implies that
    \begin{enumerate}
        \item $\AT[V_c] \Q[V_s C]$ has full row rank.
        \item $\phantom{_\mr{C V_c R}} \QT[V_s] \AIc[1] = \mb{0}$.
        \item $\phantom{_\mr{V_c R}} \QT[V_s C] \AIc[2] = \mb{0}$.
        \item $\phantom{_\mr{R}} \QT[V_s C V_c] \AIc[3] = \mb{0}$.
        \item $\phantom{_\mr{}} \QT[V_s C V_c R] \AIc[4] = \mb{0}$.
    \end{enumerate}
\end{proposition}

In addition to controlled sources, we also allow for controlled resistors, since these frequently occur in circuit models, e.g.\ in the context of power electronics. This constitutes the reason for the extra argument of $\gR$ in \cref{as:ej}. Furthermore, we introduce a shorthand for groups of incidence matrices, e.g.\ $\A[CV] \coloneqq [\A[C], \A[V]]$. Finally, we write
\begin{align}
     \ic \big( \AT\! \vphi, \iL, \iVc, t \big) = \begin{bmatrix}
        \ic[1] \big( \AT\! \vphi, \iL, \iVc, t \big)\\[0.5ex]
        \ic[2] \big( \AT[C R V] \vphi, \iL, \iVc, t \big)\\[0.5ex]
        \ic[3] \big( \AT[C R V] \vphi, \iL, t \big)\\[0.5ex]
        \ic[4] \big( \AT[C V] \vphi, \iL, t \big)
     \end{bmatrix} \label{eq:cs}
\end{align}
for the controlled current sources, where we only include those dependencies, that allow for a topological decoupling. In general, controlled current sources may depend on specific subsets of the branch voltages $\AT\! \vphi$, where $\vphi$ are the nodal potentials, as well as currents through inductors, denoted by $\iL$. Depending on their topological configuration, they may also depend on currents through controlled voltage sources, given by $\iVc$, however they may never depend on currents through independent voltage sources, which are written as $\iVs$.

Using all the introduced notation, we can write MNA in the following form
\begin{subequations}
    \label{eq:mnatd}
    \begin{alignat}{2}
            &\A[C] \dd{t} \qC \big( \AT[C] \vphi \big) &&\notag\\
            &\quad + \A[R] \gR \big( \AT[R] \vphi, \AT[C V] \vphi, \iL, t \big) &&\notag\\
            &\quad + \A[L] \iL + \A[V_s] \iVs + \A[V_c] \iVc &&\notag\\
            &\quad + \A[I_s] \is(t) + \A[I_c] \ic \big( \AT\! \vphi, \iL, \iVc, t \big) &&= \mb{0} \label{eq:mnatd.1}\\
        &\dd{t} \phiL(\iL) - \AT[L] \vphi &&= \mb{0} \label{eq:mnatd.2}\\
        &\AT[V_s] \vphi - \vs(t) &&= \mb{0} \label{eq:mnatd.3}\\
        &\AT[V_c] \vphi - \vc \big( \AT[C V_s] \vphi, \iL, t \big) &&= \mb{0}, \label{eq:mnatd.4}
    \end{alignat}
\end{subequations}
where we again only include those dependencies in expressions for controlled elements, which still allow for a topological decoupling. For resistors this implies that, outside of voltages across other resistors, they may also depend on voltages across capacitors or voltage sources, as well as currents through inductors. Controlled voltage sources on the other hand, represented by $\vc$, may only depend on voltages across capacitors or independent voltage sources. Similar to resistors, they may also depend on inductor currents, but not on any other solution variables, such as e.g.\ currents through voltage sources. Lastly, we note that $\is(t)$ and $\vs(t)$ describe independent current and voltage sources, respectively.

\section{Topological decoupling}
\label{sec:dec}
The following theorem states the main result regarding the decoupling of the equations of MNA.

\begin{theorem}[Decoupling of MNA]
    \label{thm:dec}
    Using \cref{as:wp,,as:ej,as:cs}, the system of MNA, as stated in \cref{eq:mnatd}, can be decoupled to a semi-explicit index one DAE using only basis matrices as introduced in \cref{def:bf}.
\end{theorem}

Before proving \cref{thm:dec}, we collect a few auxiliary results. These will be needed at key points during the decoupling, however their proofs are shifted to \cref{sec:app}, to not distract from the main argument.

\begin{lemma}
    \label{lem:ap}
    ~\begin{enumerate}
        \item Let $\mb{P}$ be a basis matrix of $\AT$, then $\AT \mb{P}$ has full column rank. If $\AT$ further has full row rank, then $\AT \mb{P}$ is regular.

        \item Let $\V$ be a basis matrix of $\AT$ and let $\AT$ have full column rank, then $\A \V$ is regular.

        \item Let $n \in \mathbb{N}$ and $[\mb{A}_1, \mb{A}_2, \dotsc, \mb{A}_n]$ have full row rank, then $\mb{Q}_{n-1}^\T \dotsb \mb{Q}_1^\T \mb{A}_n$ also has full row rank, where $\mb{Q}_1$ is a basis matrix of $\mb{A}_1^\T$ and $\mb{Q}_m, m \geq 2$ is a basis matrix of $\mb{A}_m^\T \mb{Q}_1 \dotsb \mb{Q}_{m-1}$.

        \item Let $\A$, $\mb{B}$ be matrices with $\A$ being positive definite and $\mb{B}$ having full column rank, then $\mb{B}^\T \A \mb{B}$ is also positive definite.

        \item Let $n \in \mathbb{N}$ and $\mb{A}_1, \mb{A}_2, \dotsc, \mb{A}_n$ be regular, then any block lower triangular matrix
        \begin{align*}
            \mb{A} = \begin{bmatrix}
                \mb{A}_1 & \mb{0} & \cdots & \mb{0}\\
                \bullet & \mb{A}_2 & \ddots & \vdots\\
                \vdots & \ddots & \ddots & \mb{0}\\
                \bullet & \cdots & \bullet & \mb{A}_n
            \end{bmatrix},
        \end{align*}
        where the $\bullet$ denote arbitrary blocks, is also regular.
    \end{enumerate}
\end{lemma}

\begin{proof}[Proof of \cref{thm:dec}]
    \label{pr:dec}
    The proof is split into four steps, the first two of which eliminate \cref{eq:mnatd.3} and \cref{eq:mnatd.4}, before dealing with \cref{eq:mnatd.2} and finally identifying a semi-explicit index one DAE, as promised. The first two steps amount to identifying and isolating constraints on the nodal potentials that are due to (independent and controlled) voltage sources. These constraints in turn imply restrictions on the currents through voltage sources, which can be handled in an analogous fashion due to the similarity of the conditions. Lastly, we isolate constraints on the currents through inductors, which arise from particular topological configurations.

    \item \emph{Eliminating \cref{eq:mnatd.3}.} Based on ideas from \cite{tischendorf1999,estevez2000} and following a strategy similar to \cite[Section 7.1.3]{jansen2014}, we begin by introducing basis matrices $\mb{P}_\mr{V_s}$ and $\Q[V_s]$ of $\AT[V_s]$ to split the nodal potentials as $\vphi = \mb{P}_\mr{V_s} \vphit[V_s] + \Q[V_s] \vphib[V_s]$. Inserting this splitting into \cref{eq:mnatd.3} gives
    \begin{align}
        g_1(\vphit[V_s], t) \coloneqq \AT[V_s] \mb{P}_\mr{V_s} \vphit[V_s] - \vs(t) = \mb{0} \label{eq:mnatd.a1}
    \end{align}
    and this equation is uniquely solvable for $\vphit[V_s]$ by \cref{prop:wp}.1 and \cref{lem:ap}.1. Note that we defined an auxiliary function $g_1$ here (in reference to the notation of \cref{def:seio}) and that we will continue to define similar functions in order to help simplify notation later on.

    Moving on with the decoupling, we also split \cref{eq:mnatd.1}, by multiplying with $\PT[V_s]$ and $\QT[V_s]$ from the left. This yields
    \begin{alignat}{2}
        &\PT[V_s] \bigg( \A[C] \dd{t} \qC \big( \AT[C] \vphi \big) &&\notag\\
        &\quad + \A[R] \gR \big( \AT[R] \vphi, \AT[C V] \vphi, \iL, t \big) &&\notag\\
        &\quad + \A[L] \iL + \A[V_c] \iVc + \A[I_s] \is(t) &&\notag\\
        &\quad + \A[I_c] \ic \big( \AT\! \vphi, \iL, \iVc, t \big) \bigg) &&\notag\\
        &\quad + \PT[V_s] \A[V_s] \iVs &&= \mb{0} \label{eq:mnatd.a2}\\
        &\QT[V_s] \bigg( \A[C] \dd{t} \qC(\cdot) + \A[R] \gR(\cdot) &&\notag\\
        &\quad + \A[L] \iL + \A[V_c] \iVc + \A[I_s] \is(t) \bigg) &&\notag\\
        &\quad + \QT[V_s] \A[I_c] \ic \big( \AT[C R V] \vphi, \iL, \iVc, t \big) &&= \mb{0}, \label{eq:mnatd.i1}
    \end{alignat}
    where we note that \cref{eq:mnatd.a2} is uniquely solvable for $\iVs$ by the same reasoning as before. We further note that the arguments of the controlled current source term in \cref{eq:mnatd.i1} changed in accordance with \cref{eq:cs} and \cref{prop:cs}.2. Here and in the following, we replace dependencies that are not central to an argument with $\cdot$ to improve readability.

    \item \emph{Eliminating \cref{eq:mnatd.4}.} In a second step, we introduce basis matrices $\mb{P}_\mr{C}$ and $\Q[C]$ of $\AT[C] \Q[V_s]$ and further split the nodal potentials as $\vphib[V_s] = \mb{P}_\mr{C} \vphit[C] + \Q[C] \vphib[C]$. We also continue to split \cref{eq:mnatd.i1}, by multiplying with $\PT[C]$ and $\QT[C]$ from the left to obtain
    \begin{alignat}{2}
        &\PT[C] \QT[V_s] \bigg( \A[C] \dd{t} \qC(\cdot) + \A[R] \gR(\cdot) &&\notag\\
        &\quad + \A[L] \iL + \A[V_c] \iVc + \A[I_s] \is(t) \bigg) &&\notag\\
        &\quad + \PT[C] \QT[V_s] \A[I_c] \ic \big( \AT[C R V] \vphi, \iL, \iVc, t \big) &&= \mb{0} \label{eq:mnatd.d1}\\
        &\QT[V_s C] \big( \A[R] \gR(\cdot) + \A[L] \iL + \A[V_c] \iVc + \A[I_s] \is(t) \big) &&\notag\\
        &\quad + \QT[V_s C] \A[I_c] \ic \big( \AT[C R V] \vphi, \iL, t \big) &&= \mb{0}, \label{eq:mnatd.i2}
    \end{alignat}
    where we again changed the arguments of the controlled current sources in \cref{eq:mnatd.i2}, in line with \cref{eq:cs} and \cref{prop:cs}.3.

    Next, we use basis matrices $\mb{P}_\mr{V_c}$ and $\Q[V_c]$ of $\AT[V_c] \Q[V_s C]$ to split the nodal potentials once more as $\vphib[C] = \mb{P}_\mr{V_c} \vphit[V_c] + \Q[V_c] \vphib[V_c]$. Inserting the entire splitting of $\vphi$ up to this point into \cref{eq:mnatd.4}, we find
    \begin{align}
        g_3(\cdot) \coloneqq\ &\AT[V_c] \Big( \mb{P}_\mr{V_s} \vphit[V_s] + \Q[V_s] \big( \mb{P}_\mr{C} \vphit[C] + \Q[C] \mb{P}_\mr{V_c} \vphit[V_c] \big) \Big) \notag\\
        &- \vc \Big( \AT[C V_s] \big( \mb{P}_\mr{V_s} \vphit[V_s] + \Q[V_s] \mb{P}_\mr{C} \vphit[C] \big), \iL, t \Big) = \mb{0}, \label{eq:mnatd.a3}
    \end{align}
    which is uniquely solvable for $\vphit[V_c]$ by \cref{prop:cs}.1 and \cref{lem:ap}.1. We also continue to split \cref{eq:mnatd.i2}. Multiplying with $\PT[V_c]$ and $\QT[V_c]$ from the left yields
    \begin{alignat}{2}
        g_5(\cdot) \coloneqq\ &\PT[V_c] \QT[V_s C] \Big( \A[R] \gR \big( \AT[R] \vphi, \AT[C V] \vphi, \iL, t \big) &&\notag\\
        &\quad + \A[L] \iL + \A[I_s] \is(t) \Big) &&\notag\\
        &\quad + \PT[V_c] \QT[V_s C] \A[I_c] \ic \big( \AT[C R V] \vphi, \iL, t \big) &&\notag\\
        &\quad + \PT[V_c] \QT[V_s C] \A[V_c] \iVc &&= \mb{0} \label{eq:mnatd.a4}\\
        &\QT[V_s C V_c] \Big( \A[R] \gR(\cdot) + \A[L] \iL + \A[I_s] \is(t) \Big) &&\notag\\
        &\quad + \QT[V_s C V_c] \A[I_c] \ic \big( \AT[C V] \vphi, \iL, t \big) &&= \mb{0}, \label{eq:mnatd.i3}
    \end{alignat}
    where we note that \cref{eq:mnatd.a4} is uniquely solvable for $\iVc$ by the same argument as before. The dependencies of the controlled current sources in \cref{eq:mnatd.i3} also changed again, this time according to \cref{eq:cs} and \cref{prop:cs}.4.

    \item \emph{Dealing with \cref{eq:mnatd.2}.} We complete the splitting of the nodal potentials by introducing basis matrices $\mb{P}_\mr{R}$ and $\Q[R]$ of $\AT[R] \Q[V_s C V_c]$ and writing $\vphib[V_c] = \mb{P}_\mr{R} \vphit[R] + \Q[R] \vphib[R]$. Multiplying \cref{eq:mnatd.i3} by $\PT[R]$ and $\QT[R]$ from the left gives
    \begin{alignat}{2}
        g_4(\cdot) \coloneqq\ &\PT[R] \QT[V_s C V_c] \Big( \A[R] \gR(\cdot) + \A[L] \iL + \A[I_s] \is(t) \Big) &&\notag\\
        &\quad + \PT[R] \QT[V_s C V_c] \A[I_c] \ic \big( \AT[C V] \vphi, \iL, t \big) &&= \mb{0} \label{eq:mnatd.a5}\\
        &\QT[V_s C V_c R] \big( \A[L] \iL + \A[I_s] \is(t) \big) &&= \mb{0}, \label{eq:mnatd.i5}
    \end{alignat}
    where we eliminated also the last controlled current source type from \cref{eq:mnatd.i5}, in line with \cref{prop:cs}.5. In a last step, we introduce basis matrices $\V[L]$ and $\W[L]$ of $ \AT[L] \Q[V_s C V_c R]$ and split the inductor currents as $\iL = \V[L] \iLt + \W[L] \iLb$. Inserting this splitting into \cref{eq:mnatd.i5}, we find
    \begin{align}
        g_2(\cdot) \coloneqq \QT[V_s C V_c R] \big( \A[L] \V[L] \iLt + \A[I_s] \is(t) \big) = \mb{0}, \label{eq:mnatd.a6}
    \end{align}
    which is uniquely solvable for $\iLt$ by \cref{prop:wp}.2, \cref{lem:ap}.3 and \cref{lem:ap}.2.

    Finally, we split \cref{eq:mnatd.2} by multiplying with $\VT[L]$ and $\WT[L]$ from the left to find
    \begin{align}
        \VT[L] \bigg( \dd{t} \phiL(\iL) - \AT[L] \vphi \bigg) &= \mb{0} \label{eq:mnatd.i6}\\
        \WT[L] \bigg( \dd{t} \phiL(\iL) - \AT[L] \vphi \bigg) &= \mb{0}. \label{eq:mnatd.d2}
    \end{align}
    Inserting the entire splitting of $\vphi$ into \cref{eq:mnatd.i6} gives
    \begin{alignat}{2}
        &\VT[L] \dd{t} \phiL(\iL) - \VT[L] \AT[L] \bigg( \mb{P}_\mr{V_s} \vphit[V_s] + \Q[V_s] \Big( \mb{P}_\mr{C} \vphit[C] &&\notag\\
        &\quad + \Q[C] \big( \mb{P}_\mr{V_c} \vphit[V_c] + \Q[V_c] \mb{P}_\mr{R} \vphit[R] \big) \Big) \bigg) &&\notag\\
        &\quad - \VT[L] \AT[L] \Q[V_s C V_c R] \vphib[R] &&= \mb{0}, \label{eq:mnatd.a7}
    \end{alignat}
    which can be seen to be uniquely solvable for $\vphib[R]$ by the same reasoning as before.

    \item \emph{Identifying the semi-explicit index one DAE.}
    We can now identify and collect the differential and algebraic parts corresponding to \cref{def:seio}. For the differential part, we take \cref{eq:mnatd.d1} and \cref{eq:mnatd.d2}, insert the variable splittings for $\vphi$ and $\iL$ at key locations and resolve the total time derivatives to find
    \begin{subequations}
        \label{eq:mnatdd}
        \begin{alignat}{2}
            &\PT[C] \QT[V_s] \bigg( \A[C] \mb{C} \big( \AT[C] \vphi \big) \AT[C] \Q[V_s] \mb{P}_\mr{C} \dd{t} \vphit[C] &&\notag\\
            &\quad + \A[C] \mb{C} \big( \AT[C] \vphi \big) \AT[C] \mb{P}_\mr{V_s} \dd{t} \vphit[V_s] &&\notag\\
            &\quad + \A[R] \gR \big( \AT[R] \vphi, \AT[C V] \vphi, \iL, t \big) &&\notag\\
            &\quad + \A[L] \iL + \A[V_c] \iVc + \A[I_s] \is(t) \bigg) &&\notag\\
            &\quad + \PT[C] \QT[V_s] \A[I_c] \ic \big( \AT[C R V] \vphi, \iL, \iVc, t \big) &&= \mb{0} \label{eq:mnatdd1}\\
            &\WT[L] \bigg( \mb{L}(\iL) \W[L] \dd{t} \iLb + \mb{L}(\iL) \V[L] \dd{t} \iLt - \AT[L] \vphi \bigg) &&= \mb{0}. \label{eq:mnatdd2}
        \end{alignat}
    \end{subequations}
    Here, $\mb{C}$ and $\mb{L}$ are defined as in \cref{as:ej} and we emphasize that both $\vphit[V_s]$ and $\iLt$ are purely time-dependent functions depending only on independent voltage and current sources, by \cref{eq:mnatd.a1} and \cref{eq:mnatd.a6} respectively. We can therefore bring \cref{eq:mnatdd} to the form of \cref{eq:seio.1} in \cref{def:seio}, with $\mb{x}^\T = [\vphit[C]^\T, \iLb^\top]$ as the differential variables, since the leading expressions (in front of time derivatives involving differential variables) are positive definite by \cref{lem:ap}.1, \cref{def:bf}, \cref{as:ej} and \cref{lem:ap}.4. Furthermore, we remark that all expressions remain sparse and that they can be assembled efficiently due to the topological nature of the basis matrices, see the proof of \cref{prop:tbf} later.

    Before we identify the algebraic variables $\mb{y}$ from \cref{def:seio}, we note that neither $\iVs$ nor $\vphib[R]$ appear in the differential part\footnote{In \cref{eq:mnatdd1} this is true as $\AT[L I] \vphi$ does not appear and in \cref{eq:mnatdd2} this is guaranteed by the definition of $\W[L]$.} \cref{eq:mnatdd}. In fact, these two variables form a third set $\mb{z}^\T = [\iVs^\T, \vphib[R]^\T]$, in the following referred to as \emph{output} variables, since they influence neither the differential nor the algebraic variables but can instead be recovered from them. Therefore, we can deal with the corresponding equations \cref{eq:mnatd.a2} and \cref{eq:mnatd.a7} separately. Written as one system\footnote{The extra arguments $\mb{x}'$ and $\mb{y}'$ indicate that $\mb{h}$ may also depend on the time derivatives of the differential and algebraic variables.} $\mb{h}(\mb{x}', \mb{y}', \mb{x}, \mb{y}, \mb{z}, t) = \mb{0}$, the corresponding Jacobian w.r.t.\ $\mb{z}$ reads
    \begin{align*}
        \pp[\mb{h}]{\mb{z}} = \begin{bmatrix}
            \PT[V_s] \A[V_s] & \\
            & -\VT[L] \AT[L] \Q[V_s C V_c R]
        \end{bmatrix},
    \end{align*}
    since \cref{eq:mnatd.a2} does not depend on $\vphib[R]$ as $\AT[L I] \vphi$ does not appear. We recall that we already showed both diagonal blocks to be regular earlier, thus the Jacobian itself is also regular by \cref{lem:ap}.5 and hence $\mb{h}$ is uniquely solvable for $\mb{z}$.

    We can now identify the algebraic variables
    \begin{align*}
        \mb{y}^\T = [\vphit[V_s]^\T, \iLt^\top, \vphit[V_c]^\T, \vphit[R]^\T, \iVc^\T]
    \end{align*}
    and therefore the algebraic part $\mb{g}$ corresponding to \cref{eq:seio.2} in \cref{def:seio} is given by
    \begin{align}
        \mb{g}(\mb{x}, \mb{y}, t) = \begin{bmatrix}
            g_1(\vphit[V_s], t)\\
            g_2(\iLt, t)\\
            g_3(\vphit[V_c], \cdot, t)\\
            g_4(\vphit[R], \cdot, t)\\
            g_5(\iVc, \cdot, t)\\
        \end{bmatrix} = \mb{0}, \label{eq:mnatda}
    \end{align}
    where we only list key dependencies for improved readability and replace the remaining ones using $\cdot$ similar to before. We emphasize that $\mb{g}$ is not allowed to depend on the output variables $\mb{z}$ and indeed, this can easily be verified as neither $\iVs$ nor $\AT[L I] \vphi$ appear in any of the functions, cf.\ \cref{eq:mnatd.a1,eq:mnatd.a6,eq:mnatd.a3,eq:mnatd.a5,eq:mnatd.a4}. The Jacobian of $\mb{g}$ w.r.t.\ $\mb{y}$ is given by
    \begin{align*}
        \pp[\mb{g}]{\mb{y}} &= \begin{bmatrix}
            \mb{A}_1 & & & & \\
            & \mb{A}_2 & & & \\
            \bullet & \bullet & \mb{A}_3 & & \\
            \bullet & \bullet & \bullet & \mb{A}_4 & \\
            \bullet & \bullet & \bullet & \bullet & \mb{A}_5
        \end{bmatrix},
    \end{align*}
    where the individual matrix blocks are defined as follows
    \begin{align*}
        \mb{A}_1 &\coloneqq \AT[V_s] \mb{P}_\mr{V_s}\\
        \mb{A}_2 &\coloneqq \QT[V_s C V_c R] \A[L] \V[L]\\
        \mb{A}_3 &\coloneqq \AT[V_c] \Q[V_s C] \mb{P}_\mr{V_c}\\
        \mb{A}_4 &\coloneqq \PT[R] \QT[V_s C V_c] \A[R] \mb{G} \big( \AT[R] \vphi, \AT[C V] \vphi, \iL, t \big) \AT[R] \Q[V_s C V_c] \mb{P}_\mr{R}\\
        \mb{A}_5 &\coloneqq \PT[V_c] \QT[V_s C] \A[V_c].
    \end{align*}
    Here, $\mb{G}$ is defined analogous to $\mb{C}$ and $\mb{L}$ before, with $\mb{x} = \AT[R] \vphi$ when comparing with \cref{as:ej}. $\mb{A}_4$ can then be shown to be regular in the same way as the expression containing $\mb{C}$ earlier. Since the remaining diagonal blocks were already proved to be regular, \cref{lem:ap}.5 yields that $\mb{g}$ is uniquely solvable for $\mb{y}$, which completes the proof.

    As a final remark, we note that the first three equations can be solved independently of the last two and that they each only require the solution of one linear system. Furthermore, we again emphasize that all products involving incidence and basis matrices remain sparse and that they can be assembled efficiently, compare also the proof of \cref{prop:tbf}.
\end{proof}

The following proposition ensures that the decoupling can be computed using only topological information about the circuit.

\begin{proposition}[Topological basis matrices]
    \label{prop:tbf}
    All basis matrices involved in the decoupling can be determined topologically, i.e.\ by only looking at the graph of the circuit.
\end{proposition}

Before proving \cref{prop:tbf}, we note the following (well-known) properties of incidence matrices of spanning trees.

\begin{lemma}[Incidence matrix of a spanning tree]
    \label{lem:imst}
    ~\begin{enumerate}
        \item The unreduced incidence matrix of a spanning tree has full column rank.
        \item A reduced incidence matrix of a spanning tree is regular.
    \end{enumerate}
\end{lemma}

\begin{proof}
    \label{pr:imst}
    See e.g.\ \cite[Lemma 2.2 and Lemma 2.7]{bapat2014}.
\end{proof}

\begin{proof}[Proof of \cref{prop:tbf}]
    \label{pr:tbf}
    We begin the proof by noting that there are three different kinds of structures for which we need to determine basis matrices. The first kind are basis matrices of incidence matrices, as e.g.\ $\mb{P}_\mr{V_s}$ and $\Q[V_s]$ of $\AT[V_s]$. The second kind are basis matrices of products of basis and incidence matrices, as e.g.\ $\mb{P}_\mr{C}$ and $\Q[C]$ of $\AT[C] \Q[V_s]$, but also all subsequent basis matrices required for splitting the nodal potentials fall into this category. The third and last kind are basis matrices of the transpose of products of basis and incidence matrices, as needed for $\V[L]$ and $\W[L]$ of $\AT[L] \Q[V_s C V_c R]$ in the final step of the decoupling. In the following, we provide topological constructions and interpretations for each of these structures.

    \item \emph{Basis matrices of incidence matrices.} We first consider a reduced incidence matrix $\A[E] \in \mathbb{R}^{n_\mr{E} \times b_\mr{E}}$ corresponding to an arbitrary element type $\mr{E}$. In order to determine $\Q[E]$, we need to find a basis for $\ker \AT[E]$, compare \cref{def:bf}. For this, we note that the transpose of the unreduced incidence matrix of an arbitrary circuit fulfilling \cref{as:wp} can be written as
    \begin{align}
        \AT = \begin{bNiceArray}{c}
                \AT[t]\\
                \cmidrule(lr){1-1}
                \AT[\ell]
            \end{bNiceArray}, \label{eq:acim}
    \end{align}
    where the upper part describes a spanning tree and the lower part the branches leading to fundamental loops, in the following simply referred to as the fundamental loops. We further note that the nonzero entries in each row of $\AT$ must sum to zero by \cref{def:im} and therefore, a basis $\Q$ for $\ker \AT$ is given by $\Q = \mb{1}_n$, with $\AT \in \mathbb{R}^{b \times n}$ and $\mb{1}_n$ a vector of $n$ ones, as $\ker \AT$ is one-dimensional by \cref{lem:imst}.1. For a reduced incidence matrix $\A'$, we instead find the trivial kernel, since the corresponding spanning tree is regular by \cref{lem:imst}.2.

    Returning to $\A[E]$ and summarizing, we can consider all the connected components ($\mr{c}_1$ through $\mr{c}_c$) of $\A[E]$ individually to end up with
    \begin{align}
        \Q[E] = \begin{bmatrix}
            \mb{1}_{|\mr{c}_1|} & & & \\
            & \mb{1}_{|\mr{c}_2|} & & \\
            & & \ddots & \\
            & & & \mb{1}_{|\mr{c}_{c-1}|}\\
            & & & \mb{0}_{|\mr{c}_c|}
        \end{bmatrix}, \label{eq:QE}
    \end{align}
    where we reordered the nodes according to the $c$ connected components of $\A[E]$ and defined the last component to contain the ground node, compare also \cite[Section 7.1.1]{jansen2014}. (Note that each node not incident to an element of type $\mr{E}$ corresponds to an individual component.)

    In order to describe $\PE$, we simply need one identity matrix of appropriate size per connected component to complete $\Q[E]$ to a basis of $\mathbb{R}^{n_\mr{E}}$. This yields
    \begin{align}
        \PE = \begin{bmatrix}
            \mb{I}_{|\mr{c}_1|-1} & & & \\[0.3ex]
            \mb{0}_{|\mr{c}_1|-1}^\T & & & \\
            & \mb{I}_{|\mr{c}_2|-1} & & \\[0.3ex]
            & \mb{0}_{|\mr{c}_2|-1}^\T & & \\
            & & \ddots & \\
            & & & \mb{I}_{|\mr{c}_c|}
        \end{bmatrix}, \label{eq:PE}
    \end{align}
    where $\mb{I}_n$ denotes an identity matrix of size $n$ and matrices of size zero are ignored. (The latter occur for each node not incident to an element of type $\mr{E}$.)

    \item \emph{Basis matrices of products.} We begin by determining the structure of a product $\ATE[2] \QE[1]$, for $\mr{E}_1$, $\mr{E}_2$ two different element types and distinguish the following three cases:
    \begin{enumerate}
        \item If both nodes of a branch (column) $\mb{b} \in \mathbb{R}^{n_{\mr{E}_2}}$ of $\AE[2]$ are in the same component of $\AE[1]$, then $\mb{b}^\T \QE[1] = \mb{0}_{q_{\mr{E}_1}}^\T$, where $\QE[1] \in \mathbb{R}^{n_{\mr{E}_1} \times q_{\mr{E}_1}}$.

        \item If one of the nodes of $\mb{b}$ is in the ground component of $\AE[1]$ while the other is not, we have $\mb{b}^\T \QE[1] = (\mb{b}')^\T$, with $\mb{b}'$ containing only one nonzero entry corresponding to the non-ground component.

        \item If the nodes are in two different non-ground components of $\AE[1]$, we again find $\mb{b}^\T \QE[1] = (\mb{b}')^\T$, however this time $\mb{b}'$ contains two nonzero entries corresponding to the respective components of $\AE[1]$.
    \end{enumerate}
    If we order the branches of $\AE[2]$ according to these cases, we can rewrite the product as
    \begin{align}
        \ATE[2] \QE[1] = \begin{bmatrix}
            \mb{B}_1^\T\\
            \cmidrule(lr){1-1}
            \mb{B}_2^\T\\[0.3ex]
            \mb{B}_3^\T
        \end{bmatrix} \QE[1] = \begin{bNiceArray}{c}
                \mb{0}\\
                \cmidrule(lr){1-1}
                (\AE[2]')^\T
            \end{bNiceArray}, \label{eq:bfp}
    \end{align}
    where $\AE[2]'$ corresponds to the graph obtained from contracting $\AE[2]$ by the branches of $\AE[1]$. Here, contracting a graph is defined as identifying the nodes of each branch that is contracted. This topological interpretation therefore corresponds to cases 2 and 3 ($\mb{B}_2$ and $\mb{B}_3$) above, as each component of $\AE[1]$ is contracted into a single node. Case 1 corresponds to branches that are contracted into self-loops and since self-loops are not well-defined in the context of circuits, compare \cref{def:im,as:wp}, they end up leading to zero rows. \Cref{fig:contracting_branches} shows an example illustrating this topological correspondence.

    \begin{figure*}
        \begin{center}
            \includegraphics[width=0.66\textwidth]{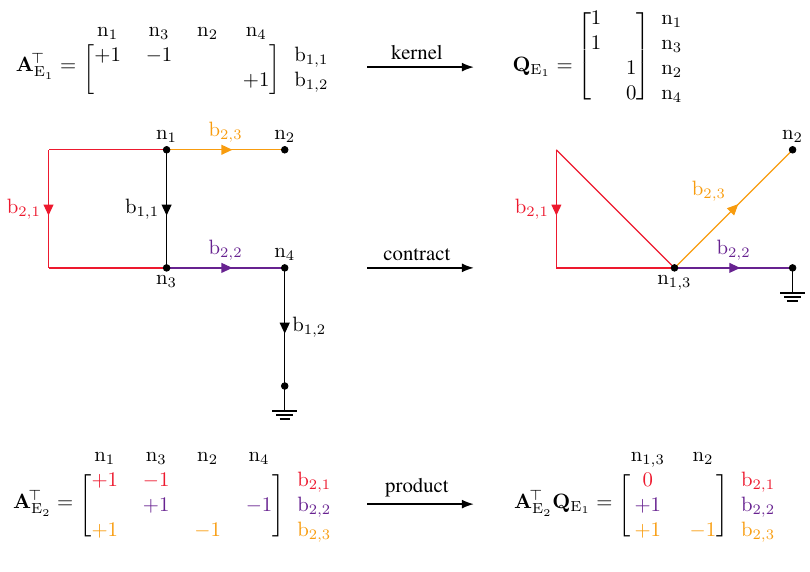}
        \end{center}
        \caption{Example illustrating branch contraction as in \cref{eq:bfp}, where the branches $\mr{b}_{2,i}$ with $1 \leq i \leq 3$ belong to case $i$ in the proof, respectively. The nodes are ordered according to the connected components of $\AE[1]$, to comply with the notation from \cref{eq:QE,eq:PE}.}
        \label{fig:contracting_branches}
    \end{figure*}

    \begin{figure*}[b]
        \begin{center}
            \includegraphics[width=0.69\textwidth]{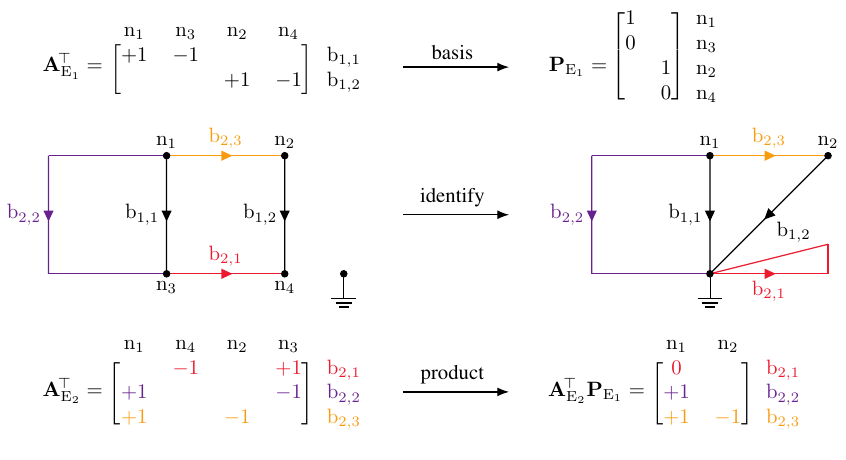}
        \end{center}
        \caption{Example illustrating (ground) node identification, where the branches $\mr{b}_{2,i}$ with $1 \leq i \leq 3$ again belong to case $i$ in the proof, respectively. The nodes are also again ordered to comply with the notation from \cref{eq:PE}.}
        \label{fig:identifying_nodes}
    \end{figure*}

    Moving on, we can now identify a basis matrix $\QE[2]$ of $\ATE[2] \QE[1]$. By noting the specific structure of \cref{eq:bfp}, we can choose $\QE[2]$ as a basis matrix of $(\AE[2]')^\T$, as this yields
    \begin{align*}
        \ATE[2] \QE[1] \QE[2] = \begin{bNiceArray}{c}
                \mb{0}\\
                \cmidrule(lr){1-1}
                (\AE[2]')^\T
            \end{bNiceArray} \QE[2] = \mb{0}.
    \end{align*}
    Since $\AE[2]'$ is the incidence matrix of a contracted graph, we can further determine $\QE[2]$ in the same topological manner as explained earlier.

    Lastly, we note that we can again reuse the case distinction to also identify the structure of $\ATE[3] \QE[1] \QE[2]$, for $\mr{E}_3$ a third element type. We immediately see that
    \begin{align*}
        \ATE[3] \QE[1] \QE[2] = \begin{bNiceArray}{c}
                \mb{0}\\
                \cmidrule(lr){1-1}
                (\AE[3]')^\T
            \end{bNiceArray} \QE[2]
    \end{align*}
    and thus, we can apply the same reasoning as before to $(\AE[3]')^\T \QE[2]$, i.e.\ to the already contracted graph $\AE[3]'$. As this leaves the general structure from \cref{eq:bfp} unchanged, all products of basis and incidence matrices of this kind can be handled this way.

    Before characterizing transposed products, we briefly look at the structure of $\ATE[2] \PE[1]$. This is e.g.\ necessary for efficiently assembling the final systems as mentioned in the proof of \cref{thm:dec}. We distinguish three cases:
    \begin{enumerate}
        \item If both nodes of a branch $\mb{b}$ of $\AE[2]$ correspond to last nodes (in the sense of the ordering of \cref{eq:PE}) of non-ground components of $\AE[1]$, then $\mb{b}^\T \PE[1] = \mb{0}_{p_{\mr{E}_1}}^\T$, where $\PE[1] \in \mathbb{R}^{n_{\mr{E}_1} \times p_{\mr{E}_1}}$.

        \item If one of the nodes of $\mb{b}$ corresponds to the last node of a non-ground component of $\AE[1]$ while the other does not, we have $\mb{b}^\T \PE[1] = (\mb{b}')^\T$, with $\mb{b}'$ containing only one nonzero entry corresponding to the latter node.

        \item If neither node corresponds to the last node of a non-ground component of $\AE[1]$, we again find $\mb{b}^\T \PE[1] = (\mb{b}')^\T$, however $\mb{b}'$ now contains two nonzero entries corresponding to the respective nodes.
    \end{enumerate}
    Overall, this gives rise to the same structure as in \cref{eq:bfp}
    \begin{align*}
        \ATE[2] \PE[1] = \begin{bNiceArray}{c}
                \mb{0}\\
                \cmidrule(lr){1-1}
                (\AE[2]')^\T
            \end{bNiceArray},
    \end{align*}
    where this time $\AE[2]'$ corresponds to the graph obtained by identifying all last nodes of non-ground components of $\AE[1]$ with the ground node. This topological correspondence is also illustrated by an example, see \cref{fig:identifying_nodes}.

    \begin{figure*}[b]
        \begin{center}
            \includegraphics[width=0.66\textwidth]{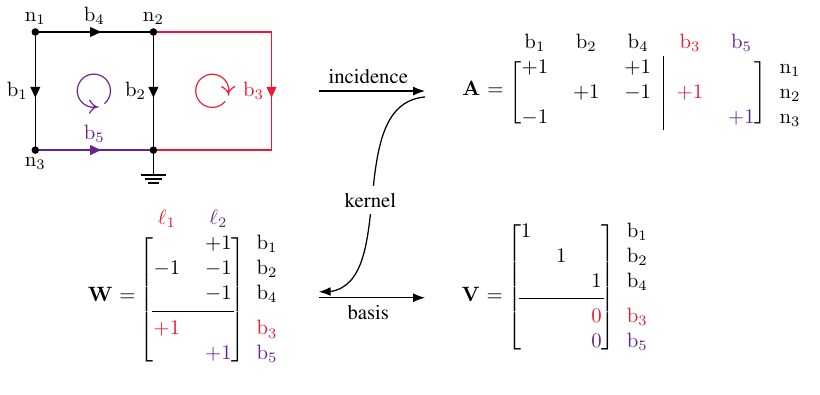}
        \end{center}
        \caption{Example illustrating the basis matrices $\V$ and $\W$ of an incidence matrix.}
        \label{fig:transposed_basis}
    \end{figure*}

    \item \emph{Basis matrices of transposed products.} In a first step, we determine a basis matrix $\W[E]$ of an incidence matrix $\A[E] \in \mathbb{R}^{n_\mr{E} \times b_\mr{E}}$ by looking at $\ker \A[E]$, compare again \cref{def:bf}. For this, we once more consider the structure from \cref{eq:acim} for the unreduced incidence matrix of an arbitrary circuit fulfilling \cref{as:wp}, $\A = \begin{bNiceArray}{c|c}
        \A[t] & \A[\ell]
    \end{bNiceArray}$. From \cref{lem:imst}.1, we already know that the kernel of $\A$ must have dimension $|\ell|$, where $|\ell|$ denotes the number of fundamental loops. Further noting that each column of $\A[\ell]$, i.e.\ each fundamental loop, connects exactly two nodes of the tree $\A[t]$ such that a (fundamental) loop is formed, we can identify the following definition for the entries of a basis $\W$ of $\A$
    \begin{align}
        w_{i,j} \coloneqq \begin{cases}
            +1, & \parbox{0.27\textwidth}{if branch $i$ and fundamental loop\\[-0.3ex] $j$ have the same orientation,}\\[2ex]
            -1, & \parbox{0.27\textwidth}{if branch $i$ and fundamental loop\\[-0.3ex] $j$ have opposite orientation,}\\[1.2ex]
            0, & \text{if branch $i$ is not part of loop $j$,}
        \end{cases} \label{eq:W}
    \end{align}
    compare again \cite[Section 7.1.1]{jansen2014}. Here, the loop direction is in principle arbitrary, but we choose the convention of fixing it based on the orientation of the fundamental loop. We verify that this definition truly ensures $\A \W = \mb{0}$, by noting that each node (that is part of a loop) appears exactly twice per loop, with the signs of the corresponding entries in $\W$ depending on the orientations of the respective branches in such a way, that their contributions sum to zero. (If the respective branches have the same orientation e.g., one branch will leave the node while the other will enter it, leading to a cancellation.) Furthermore, the columns of $\W$ are all linearly independent, as there is precisely one column containing a nonzero entry for each fundamental loop. Finally, a simple choice for $\V$ is given by
    \begin{align}
        \V = \begin{bNiceArray}{c}
            \mb{I}_{|t|}\\
            \cmidrule(lr){1-1}
            \mb{0}
        \end{bNiceArray}, \label{eq:V}
    \end{align}
    where $|t|$ denotes the size of the tree. The definitions of $\V$ and $\W$ are again best illustrated using an example, see \cref{fig:transposed_basis}.

    In the general case, where $\A[E]$ contains $c$ connected components, we can once again consider each component individually to end up with the following basis matrices
    \begin{align*}
        \W[E] = \begin{bmatrix}
            \Wc[1] & & \\
            & \ddots &\\
            & & \Wc[c]
        \end{bmatrix}, \quad \V[E] = \begin{bmatrix}
            \Vc[1] & & \\
            & \ddots &\\
            & & \Vc[c]
        \end{bmatrix},
    \end{align*}
    where we reordered the nodes according to the connected components, as in \cref{eq:QE}. We also note that the ground component requires no special treatment this time, as the derivation relies solely on the full column rank of the incidence matrix of a spanning tree, cf.\ \cref{lem:imst}. Furthermore, the same construction still works when considering basis matrices $\VE[2]$, $\WE[2]$ of products of the form\footnote{It suffices to only consider this case, as we already proved all products of interest to be of the same general form.} $\QTE[1] \AE[2]$, since
    \begin{align*}
        \QTE[1] \AE[2] \WE[2] = \begin{bNiceArray}{c|c}
            \mb{0} & \AE[2]'
        \end{bNiceArray} \WE[2]
    \end{align*}
    and hence we can choose
    \begin{align*}
        \WE[2] = \begin{bmatrix}
            \mb{I} & \\
            & \WE[2]'
        \end{bmatrix}, \quad \VE[2] = \begin{bNiceArray}{c}
            \mb{0}\\
            \cmidrule(lr){1-1}
            \VE[2]'
        \end{bNiceArray},
    \end{align*}
    where $\mb{I}$ is an identity matrix of appropriate size, $\WE[2]'$ is a basis matrix of the contracted graph $\AE[2]'$, compare \cref{eq:bfp}, and $\VE[2]'$ is a basis matrix corresponding to $\WE[2]'$.

    Finally, we conclude that all products of incidence and basis matrices can be represented topologically. They therefore remain sparse and efficient to assemble, fulfilling our promise from the proof of \cref{thm:dec}.
\end{proof}

The topological nature of the decoupling further lends itself to directly computing sets of differential and algebraic variables, as e.g.\ required for the approach described in \cite{forster2023c}.

\begin{corollary}[Topological splitting]
    \label{cor:ts}
    Using the topological basis matrices from \cref{prop:tbf}, we can explicitly determine sets of differential and algebraic variables by only looking at the graph of the circuit.
\end{corollary}

\begin{proof}
    We begin by considering a generic splitting of the nodal potentials $\vphi = \mb{P} \vphit + \Q \vphib$ with $\vphi \in \mathbb{R}^n$. Following a similar strategy as in the proof of \cref{prop:tbf}, we first look at a single non-ground component and find
    \begin{align*}
        \vphi = \begin{bmatrix}
            \mb{I}_{n-1}\\
            \mb{0}_{n-1}^\T
        \end{bmatrix} \vphit + \mb{1}_n \vphib,
    \end{align*}
    compare also \cref{eq:QE,eq:PE}. This implies that we can choose the algebraic variables as $\vphib = \varphi_n$, while the differential variables are then given by
    \begin{align}
        \vphit = \begin{bmatrix}
            \varphi_1 - \varphi_n\\
            \vdots\\
            \varphi_{n-1} - \varphi_n
        \end{bmatrix}. \label{eq:tsphi}
    \end{align}
    In general, the algebraic variable can be chosen as the potential corresponding to the last node, given the respective ordering of the component.

    When considering arbitrarily many components, $\vphib$ contains the potentials corresponding to the last nodes of each non-ground component, while $\vphit$ contains differences as in \cref{eq:tsphi} for the non-ground components and simply all potentials of the ground component, compare again \cref{eq:PE}. This also works recursively, such that $\vphib$ may be split further using the same considerations, as it only contains the potentials corresponding to individual nodes.

    \item We now look at a generic current splitting $\mb{i} = \V \mbt{i} + \W \mbb{i}$. For a single connected component, \cref{eq:V} yields
    \begin{align*}
        \mb{i} = \begin{bNiceArray}{c}
            \mb{I}_{|t|}\\
            \cmidrule(lr){1-1}
            \mb{0}
        \end{bNiceArray} \mbt{i} + \W \mbb{i},
    \end{align*}
    where we note that all fundamental loops correspond to zero entries in $\V$ by construction. Furthermore, we can identify the following general structure for $\W$ from \cref{eq:W}, compare also the example of \cref{fig:transposed_basis},
    \begin{align*}
        \W =
        \begin{bNiceArray}{c}
            \bullet\\
            \cmidrule(lr){1-1}
            \mb{I}_{|\ell|}
        \end{bNiceArray}.
    \end{align*}
    Hence we can choose the algebraic variables as
    \begin{align*}
        \mbb{i} = \begin{bmatrix}
            i_{|t| + 1}\\
            \vdots\\
            i_{|t| + |\ell|}
        \end{bmatrix}.
    \end{align*}
    For the differential variables, we then find
    \begin{align*}
        \mbt{i} = \begin{bmatrix}
            i_1\\
            \vdots\\
            i_{|t|}
        \end{bmatrix} - \W \mbb{i},
    \end{align*}
    which can be efficiently evaluated using \cref{eq:W}.

    When dealing with arbitrarily many components, we can simply apply the same reasoning for each component separately, i.e.\ $\mbb{i}$ will contain all the fundamental loop currents, while $\mbt{i}$ will consist of differences between the tree currents and (local) fundamental loop contributions.
\end{proof}

Using the notation and insights derived from proving \cref{prop:tbf}, we can now also provide a proof for \cref{prop:cs}.

\begin{proof}[Proof of \cref{prop:cs}]
    \label{pr:cs}
    ~\begin{enumerate}
        \item[1)] \Cref{fig:cv_loops} provides an abstract illustration of $\mr{C}$-$\mr{V}$ loops, where any number (including zero) of the elements depicted in gray may form a loop in any order with the controlled voltage source highlighted in black. As \cref{as:cs}.1 forbids such loops, the same reasoning that lead to \cref{eq:bfp} now yields
        \begin{align*}
            \AT[V_c] \Q[V_s C] = (\A[V_c]')^\T,
        \end{align*}
        since contracting $\A[V_c]$ by the branches in $\A[V_s]$ and $\A[C]$ cannot lead to zero rows, as there must always be some element of a different type in any path connecting the terminals of a controlled voltage source. Furthermore, \cref{as:cs}.1 also implies that the contracted graph $\A[V_c]'$ cannot contain a loop made up of only controlled voltage sources either. Thus, each connected component is a spanning tree and hence $(\A[V_c]')^\T$ has full row rank by \cref{lem:imst}.

        \item[2--4)] We consider cases 2 through 4 together, as the corresponding arguments all follow the same structure. \Cref{fig:cv_only_paths} shows an illustration of the paths required according to \cref{as:cs}, with case 2 requiring a $\mr{V_s}$-only path, while case 3 also allows for capacitors and case 4 further allows for controlled voltage sources. In each case, we note that the existence of the specific paths implies that contracting the $\AIc[i]$, $1 \leq i \leq 3$ by the element types present in the respective paths, leads to the contracted graphs $\AIc[i]'$ consisting only of self-loops. This proves the claims.

        \item[5)] \Cref{fig:li_cutsets} provides an abstract illustration of $\mr{L}$-$\mr{I}$ cutsets, where again any number (including zero) of the elements depicted in gray may form a loop in any order with the controlled current source highlighted in black. The important difference to case 1 is that there cannot be a path containing other element types between the terminals of the highlighted controlled current source, for the depicted subcircuit to form an $\mr{L}$-$\mr{I}$ cutset. As \cref{as:cs}.5 forbids such cutsets, it implies that there exists a $\mr{C}$-$\mr{R}$-$\mr{V}$-only path between the terminals of any controlled current source in $\AIc[4]$. The statement then follows from the same argument as for cases 2 through 4 earlier.
    \end{enumerate}
    \vspace{-\baselineskip}
\end{proof}

Before we move on to some examples, we briefly note (without proof) that the decoupling result from \cref{thm:dec} can also be transferred to flux-charge MNA, a variant of MNA that is often used in practice due to its favorable numerical properties \cite[Section I.4]{gunther2005}.

\begin{theorem}[Decoupling of flux-charge MNA]
    Under the assumptions of \cref{thm:dec}, one can derive an analogous decoupling for the flux-charge formulation of MNA, compare also \cite[Theorem 4.2]{estevez2000} for a corresponding index result.
\end{theorem}

\section{Examples}
\label{sec:ex}
We now present a number of examples illustrating the decoupling. The examples show that the assumptions of the decoupling are fulfilled by many common subcircuits and that in case \cref{as:cs} is violated, minor changes to the circuit may suffice to still make the decoupling possible. Furthermore, the assumptions can serve as a guide to what changes should be made. This guiding feature can be attributed to the locality of the topological conditions, compare also \cite{estevez2000}.

\begin{example}[Buck converter]
    {\normalfont We begin with a buck converter, an example from power electronics, see e.g.\ \cite{erickson2020} for further background. \Cref{fig:buck} shows the entire circuit at the top and the intermediate stages of the decoupling below it. The expressions next to the arrows always indicate the respective subgraphs highlighted in black directly below, compare also the topological interpretation of \cref{eq:bfp}. Grayed out parts do not appear in the respective subgraphs, while red parts indicate branches that are contracted away but still contribute zero rows, see again \cref{eq:bfp}. We also note that voltages across elements can be written as $\mb{v} = \AT\! \vphi$ in MNA, as it is common to express nonlinear dependencies in terms of voltages rather than nodal potentials in practice. Aside from the voltage controlled resistor $g_\mr{S}(v_\mr{R})$, we assume the diode to be modeled by a nonlinear resistor $g_\mr{D}(v_\mr{D})$.

    \begin{figure}[t]
        \begin{center}
            \includegraphics[width=0.47\textwidth]{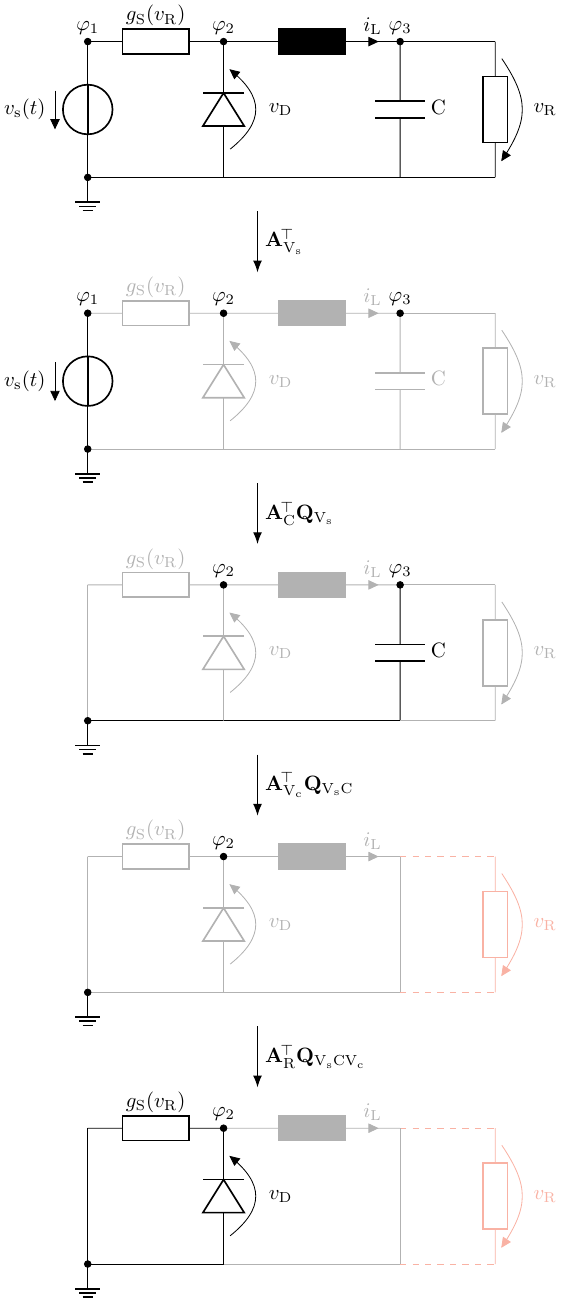}
        \end{center}
        \caption{Buck converter with simple voltage control, modeled by a voltage controlled resistor $g_\mr{S}(v_\mr{R})$.}
        \label{fig:buck}
    \end{figure}

    \begin{figure*}[b]
        \begin{center}
            \includegraphics[width=\textwidth]{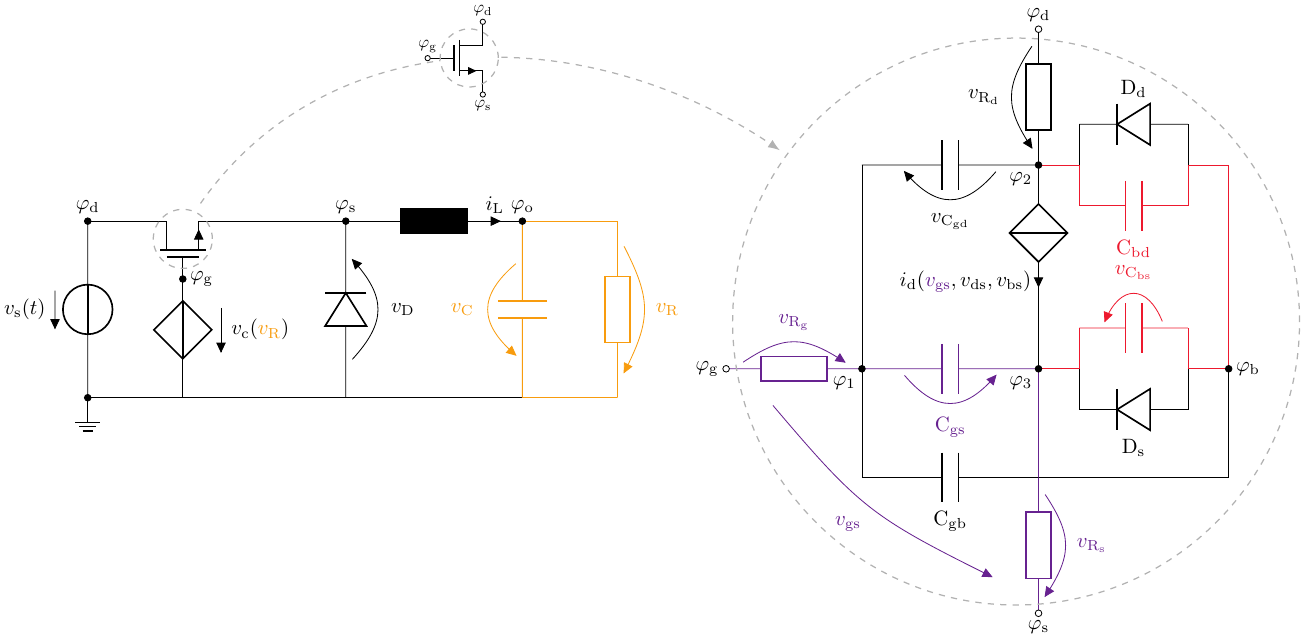}
        \end{center}
        \caption{Buck converter with a MOSFET model and a controlled voltage source replacing the voltage controlled resistor from \cref{fig:buck}. The subcircuit on the right shows the internals of the MOSFET model.}
        \label{fig:buck_mosfet}
    \end{figure*}

    In order to compare the topological approach with the algebraic perspective, we also walk through the decoupling in terms of matrices. For this, we first collect the incidence matrix of the circuit, ordered by element types
    \begin{align*}
        \A &= \begin{bNiceArray}[first-row, last-col]{c|c|ccc|c}
            \mr{C} & \mr{L} & g_\mr{s} & g_\mr{D} & \mr{R} & \mr{V_s} & \\
            & & +1 & & & +1 & \varphi_1\\
            & +1 & -1 & -1 & & & \varphi_2\\
            +1 & -1 & & & +1 & & \varphi_3
        \end{bNiceArray}.
    \end{align*}
    With this, we can now successively write down the basis matrices as derived from the topologies of the subgraphs highlighted in \cref{fig:buck}, while simultaneously verifying that they indeed fulfill \cref{def:bf}. For the first set of basis matrices corresponding to the second graph in \cref{fig:buck}, we find
    \begin{align*}
        \Q[V_s] = \begin{bNiceArray}[last-col]{cc}
            0 & & \varphi_1\\
            1 & & \varphi_2\\
            & 1 & \varphi_3
        \end{bNiceArray}, \quad \mb{P}_\mr{V_s} = \begin{bmatrix}
            1\\
            0\\
            0
        \end{bmatrix},
    \end{align*}
    which leads to
    \begin{align*}
        \vphib[V_s] = \begin{bmatrix}
            \varphi_2\\
            \varphi_3
        \end{bmatrix}, \quad \vphit[V_s] = [\varphi_1]
    \end{align*}
    as the initial variable splitting. Note that there is no need to follow the node ordering of \cref{eq:QE} in general, which is why we explicitly state the ordering of $\Q[V_s]$. Moving on to the third graph of \cref{fig:buck}, we find
    \begin{align*}
        \Q[C] = \begin{bmatrix}
            1\\
            0
        \end{bmatrix}, \quad \mb{P}_\mr{C} = \begin{bmatrix}
            0\\
            1
        \end{bmatrix}
    \end{align*}
    for the basis matrices and
    \begin{align*}
        \vphib[C] = [\varphi_2], \quad \vphit[C] = [\varphi_3]
    \end{align*}
    for the variable splitting. We note that at this point, the resistor $\mr{R}$ is contracted away alongside the capacitor $\mr{C}$, however it may still contribute a zero row, as mentioned earlier.

    The subsequent step also illustrates what happens if there are no elements of a particular type. In this case, the graph simply remains the same and thus all variables become algebraic variables of that type
    \begin{alignat*}{2}
        \Q[V_c] &= 1, &\quad \mb{P}_\mr{V_c} &= 0\\
        \vphib[V_c] &= [\varphi_2], &\quad \vphit[V_c] &= [].
    \end{alignat*}
    Finally, we complete the variable splitting for $\vphi$ with
    \begin{alignat*}{2}
        \Q[R] &= [0], &\quad \mb{P}_\mr{R} &= [1]\\
        \vphib[R] &= [], &\quad \vphit[R] &= [\varphi_2],
    \end{alignat*}
    compare also the last graph in \cref{fig:buck}.

    The same graph also tells us that the inductor $\mr{L}$ is contracted away alongside the remaining resistive elements, implying that the entire graph is contracted into the ground node. This leaves us with a single zero row in $\QT[V_s C V_c R] \A[L]$, such that the inductor currents are split as follows
    \begin{alignat*}{2}
        \W[L] &= [1], &\quad \V[L] &= [0]\\
        \iLb &= [i_\mr{L}], &\quad \iLt &= [].
    \end{alignat*}
    This yields the following system for the differential part
    \begin{align*}
        C \dd{t} \varphi_3 + G \varphi_3 - i_\mr{L} &= 0\\
        L \dd{t} i_\mr{L} - (\varphi_2 - \varphi_3) &= 0,
    \end{align*}
    where $G = 1/R$, while the algebraic part is given by
    \begin{align}
        \begin{split}
            \varphi_1 - v_\mr{s}(t) &= 0\\
            -g_\mr{S}(\varphi_3) (\varphi_1 - \varphi_2) - g_\mr{D}(-\varphi_2) (-\varphi_2) + i_\mr{L} &= 0.
        \end{split} \label{eq:bucka}
    \end{align}
    Lastly, we note that the current through the voltage source is an output variable and that it can be determined using
    \begin{align}
         g_\mr{S}(\varphi_3) (\varphi_1 - \varphi_2) + i_\mr{V_s} &= 0. \label{eq:bucko}
    \end{align}

    The example also serves to illustrate how the decoupling may be used for generating consistent initial conditions, as mentioned in \cref{sec:int}. In order to achieve this, one first determines sets of differential and algebraic variables according to \cref{cor:ts}. Afterwards, the initial conditions of the differential variables can be chosen freely, while the algebraic and output variables are obtained by solving \cref{eq:mnatda}, \cref{eq:mnatd.a2,eq:mnatd.a7}. For the buck converter from \cref{fig:buck}, this amounts to identifying $\varphi_3$ and $i_\mr{L}$ as differential variables and choosing arbitrary values for these, before solving \cref{eq:bucka,eq:bucko} to finally obtain consistent initial conditions.}
\end{example}

\begin{example}[MOSFET]
    {\normalfont For the second example, we again consider a buck converter, however this time we replace the controlled resistor from \cref{fig:buck} by a MOSFET together with a controlled voltage source. The new circuit can be seen on the left of \cref{fig:buck_mosfet}, while the right shows the (rotated) MOSFET model, compare also \cite[Fig.\ 3.37]{tietze2008} or \cite[Fig.\ 2]{estevez2000}. As this circuit already leads to large system matrices, we will only verify that it can be decoupled by looking at \cref{as:cs} and the dependencies for controlled sources listed in \cref{eq:cs,eq:mnatd}. Nonetheless, the resulting matrices are available at \cite{forster2026c}.

    We begin with the controlled current source of the MOSFET model on the right, as it represents the more interesting and instructive case. Comparing \cref{as:cs}.3 with the subcircuit, we note that there is a $\mr{C}$-only path between the terminals of the controlled current source, as marked in red in \cref{fig:buck_mosfet}. Moving on to \cref{eq:cs}, we note that the controlling voltages must be contained in $\AT[C R V] \vphi$ and observe that they can indeed all be represented using only voltages across capacitors and resistors. For $v_\mr{gs}$, this correspondence is highlighted in purple in \cref{fig:buck_mosfet} and it can be expressed as
    \begin{align*}
        v_\mr{gs} &= \varphi_\mr{G} - \varphi_\mr{S} = (\varphi_\mr{g} - \varphi_1) + (\varphi_1 - \varphi_3) + (\varphi_3 - \varphi_\mr{s})\\
        &= v_\mr{R_g} + v_\mr{C_{gs}} + v_\mr{R_s}.
    \end{align*}
    We emphasize that this kind of equivalence can be guaranteed topologically, in this specific case e.g.\ by finding a $\mr{C}$-$\mr{R}$-$\mr{V}$-only path between the nodes defining the controlling voltage. As such, this step can be efficiently automated. For completeness, we also provide similar expressions for the other controlling voltages $v_{ds}$ and $v_{bs}$
    \begin{align*}
        v_\mr{ds} &= v_\mr{R_d} + v_\mr{C_{gd}} + v_\mr{C_{gs}} + v_\mr{R_s}\\
        v_\mr{bs} &= v_\mr{C_{bs}} + v_\mr{R_s}.
    \end{align*}
    In total, this shows that the MOSFET model on the right will always lead to a circuit that can be decoupled, so long as there are no other controlled sources present that violate \cref{as:cs} or the dependencies from \cref{eq:cs} or \cref{eq:mnatd}. In particular, this means that the buck converter on the left of \cref{fig:buck_mosfet} can be decoupled, as the remaining controlled voltage source fulfills \cref{as:cs}.1 and its controlling voltage can be expressed as $v_\mr{R} = v_\mr{C}$, which can be represented using $\AT[C V_s] \vphi$ as marked in orange in \cref{fig:buck_mosfet}.}
\end{example}

\begin{example}[OPAMP]
    {\normalfont As a further example, we consider an idealized OPAMP model as shown in \cref{fig:opamp}. Commonly, one assumes that the input impedance $\mr{Z_i}$ is given by a very large resistance, while the output impedance $\mr{Z_o}$ corresponds to a very small resistance. (In the limit case, the resistances go to infinity and zero, respectively.) However it is also possible to consider an additional reactive part, see e.g.\ \cite[Fig.\ 2]{hanson2025}.}

    \begin{figure}[b]
        \begin{center}
            \includegraphics[width=0.41\textwidth]{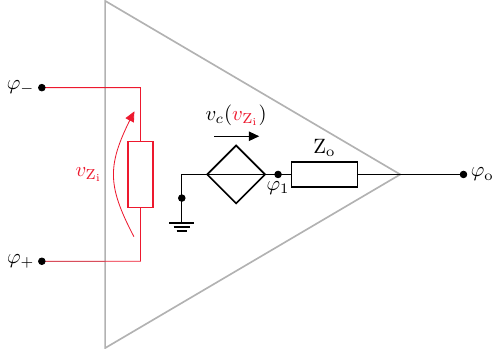}
        \end{center}
        \caption{Idealized OPAMP model.}
        \label{fig:opamp}
    \end{figure}

    {\normalfont Looking at the controlled voltage source in the center of \cref{fig:opamp}, we see that it depends on the voltage $v_\mr{Z_i}$ across the input impedance, as marked in red. In terms of the decoupling, it is therefore necessary to e.g.\ consider the input impedance to be made up of a capacitor in parallel with a resistor, such that the voltage $v_\mr{Z_i}$ may be represented using only $\AT[C V_s] \vphi$, as explained earlier. In practice, OPAMP models are often much more complex and built up from multiple MOSFETs, a common model of which we already investigated in the previous example.}
\end{example}

\begin{figure*}
    \begin{center}
        \includegraphics[width=0.6\textwidth]{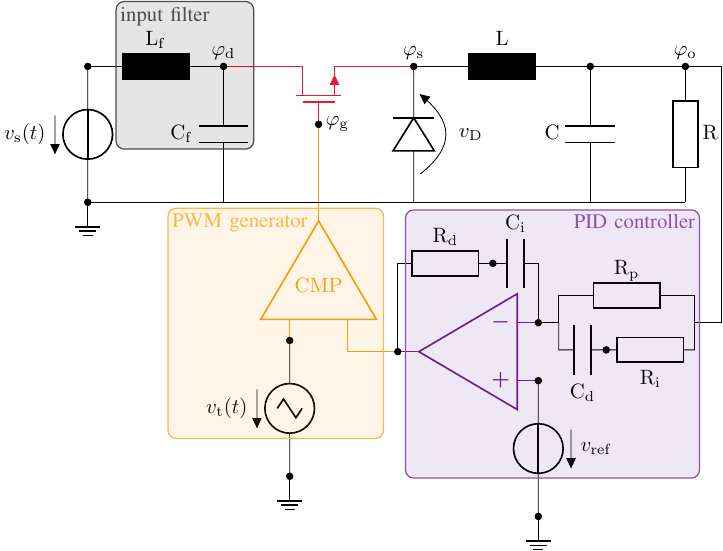}
    \end{center}
    \caption{Voltage controlled switched-mode power supply model based on the buck converter from \cref{fig:buck_mosfet}.}
    \label{fig:smps}
\end{figure*}

\begin{example}[SMPS]
    {\normalfont As a final example, we extend the buck converter from \cref{fig:buck_mosfet} to a more accurate model of a voltage controlled switched-mode power supply (SMPS), as shown in \cref{fig:smps}. The model consists of four distinct parts. An input filter highlighted in gray, a PID controller based on the output voltage highlighted in purple, a PWM generator highlighted in orange\footnote{In practice, the PWM generator is followed by a gate driver for the MOSFET, which we leave out for simplicity.} and the actual converter, in this case a buck converter. From the previous two examples, we already know that neither the MOSFET marked in red, nor the OPAMP highlighted in purple present a hindrance to the decoupling. Taking into account that the comparator (CMP), marked in orange, can also be modeled using the OPAMP circuit from \cref{fig:opamp}, we therefore see that also the entire SMPS circuit from \cref{fig:smps} may be decoupled. This further emphasizes that arbitrarily complex systems can be decoupled, so long as their constituent parts do not violate any of the assumptions. Lastly, we remark that the corresponding matrices can again be found at \cite{forster2026c}.}
\end{example}

\section{Conclusion}
\label{sec:con}
We introduced a topological decoupling for the system of MNA including controlled sources, that is valid under \cref{as:cs} and allows for the dependencies of controlled sources specified in \cref{eq:cs,eq:mnatd}. We showed that the decoupling preserves the structure of MNA in terms of sparsity and positive-definiteness and that, due to the constructive nature of the proof, the decoupling can be efficiently implemented. Using examples, we also demonstrated that the assumptions are general enough, similar to those of \cite[Theorem 4.1]{estevez2000}, such that most common circuits can be decoupled, possibly after some minor modifications.

From an applied perspective, the decoupling lends itself to various applications in the contexts of model order reduction and machine learning, however it may also find use in computing consistent initial conditions or speeding up conventional simulations. From a more theoretical perspective, our analysis did not consider controlled sources depending on time derivatives of solution variables or terms including time delays. Both phenomena appear in practice and have already been investigated with regards to their mathematical structure, see e.g.\ \cite{cortes_garcia2020} in the context of field-circuit coupling or \cite{trenn2019} in systems control. Thus, it may be of interest to further extend the decoupling in these directions. Finally, as remarked in \cref{sec:pre}, we assumed the topology of the circuit to be fixed. As ideal switches frequently occur in models of, e.g., power electronic systems however, it may also be interesting to extend the decoupling in this direction using ideas related to \cite{riaza2024a, riaza2024b}.

\section*{Acknowledgments}
This work is supported by the Graduate School CE within the Centre for Computational Engineering at Technische Universität Darmstadt and the ECSEL Joint Undertaking (JU) under grant agreement No. 101007319. The JU receives support from the European Union's Horizon 2020 research and innovation programme and the Netherlands, Hungary, France, Poland, Austria, Germany, Italy and Switzerland. Note that this work only reflects the authors' views and that the JU is not responsible for any use that may be made of the information it contains.

\bibliographystyle{IEEEtran}
\bibliography{biblio.bib}

\appendix
\label[appendix]{sec:app}
\begin{proof}[Proof of \cref{lem:ap}]
    \label{pr:ap}
    ~\begin{enumerate}
        \item Let $\mb{x} \neq \mb{0}$, then $\mb{y} \coloneqq \mb{P} \mb{x} \neq \mb{0}$, since $\mb{P}$ has full column rank as it is part of a basis. Furthermore, we have $\mb{y} \in \im \mb{P}$ and since $\im \mb{P} \cap \ker \AT = \{ \mb{0} \}$ by \cref{def:bf}, it follows that $\AT \mb{y} = \AT \mb{P} \mb{x} \neq \mb{0}$. If $\AT \in \mathbb{R}^{b \times n}$ has full row rank, we further find $\dim(\ker \AT) = n - b$, which implies $\mb{P} \in \mathbb{R}^{n \times b}$ by \cref{def:bf} and hence $\AT \mb{P} \in \mathbb{R}^{b \times b}$.

        \item This can be proved analogous to 1), by recognizing that $\V$ plays the same role for $\A$ as $\mb{P}$ does for $\AT$.

        \item By the definition of the basis matrices we have
        \begin{align}
            \begin{bmatrix}
                \mb{A}_1^\T\\
                \mb{A}_2^\T\\
                \vdots\\
                \mb{A}_n^\T
            \end{bmatrix} \mb{Q}_1 \cdots \mb{Q}_{n-1} &= \begin{bmatrix}
                \mb{0}\\
                \mb{A}_2^\T \mb{Q}_1\\
                \vdots\\
                \mb{A}_n^\T \mb{Q}_1
            \end{bmatrix} \mb{Q}_2 \cdots \mb{Q}_{n-1} \notag\\
            &\hspace{0.5em} \vdots \notag\\
            &= \begin{bmatrix}
                \mb{0}\\
                \vdots\\
                \mb{0}\\
                \mb{A}_n^\T \mb{Q}_1 \cdots \mb{Q}_{n-1}
            \end{bmatrix}. \label{eq:ap}
        \end{align}
        Now let $\mb{x}_{n-1} \neq \mb{0}$, then it holds that $\mb{x}_{k-1} \coloneqq \mb{Q}_k \mb{x}_k \neq \mb{0}$ for $1 \leq k \leq n-1$, since all $\mb{Q}_k$ have full column rank. This gives
        \begin{align*}
            \begin{bmatrix}
                \mb{A}_1^\T\\
                \mb{A}_2^\T\\
                \vdots\\
                \mb{A}_n^\T
            \end{bmatrix} \mb{Q}_1 \cdots \mb{Q}_{n-1} \mb{x}_{n-1} = \begin{bmatrix}
                \mb{A}_1^\T\\
                \mb{A}_2^\T\\
                \vdots\\
                \mb{A}_n^\T
            \end{bmatrix} \mb{x}_0 \neq \mb{0},
        \end{align*}
        as $[\mb{A}_1, \mb{A}_2, \dotsc, \mb{A}_{n-1}]^\T$ has full column rank by assumption. But then $\mb{A}_n^\T \mb{Q}_1 \cdots \mb{Q}_{n-1}$ must also have full column rank by comparing with \cref{eq:ap}.

        \item This follows e.g.\ from \cite{forster2023c}[Remark 1] or \cite{tischendorf1999}[Lemma 2.2].

        \item For $n=2$, the claim follows from
        \begin{align*}
            \det(\mb{A}) = \det \begin{bmatrix}
                \mb{A}_1 & \mb{0}\\
                \bullet & \mb{A}_2
            \end{bmatrix} = \det(\mb{A}_1) \det(\mb{A}_2).
        \end{align*}
        For arbitrary $n \in \mathbb{N}$, it can be proved using induction by rewriting $\mb{A}$ in the following way
        \begin{align*}
            \mb{A} &= \begin{bmatrix}
                \mb{A}_1 & \mb{0} & \cdots & \mb{0}\\
                \bullet & \mb{A}_2 & \ddots & \vdots\\
                \vdots & \ddots & \ddots & \mb{0}\\
                \bullet & \cdots & \bullet & \mb{A}_n
            \end{bmatrix} = \begin{bmatrix}
                \mb{A}' & \mb{0}\\
                \bullet & \mb{A}_n
            \end{bmatrix},
        \end{align*}
        where
        \begin{align*}
            \mb{A}' &= \begin{bmatrix}
                \mb{A}_1 & \mb{0} & \cdots & \mb{0}\\
                \bullet & \mb{A}_2 & \ddots & \vdots\\
                \vdots & \ddots & \ddots & \mb{0}\\
                \bullet & \cdots & \bullet & \mb{A}_{n-1}
            \end{bmatrix}.
        \end{align*}
    \end{enumerate}
\end{proof}

\begin{IEEEbiography}[{\includegraphics[width=1in,height=1.25in,clip,keepaspectratio]{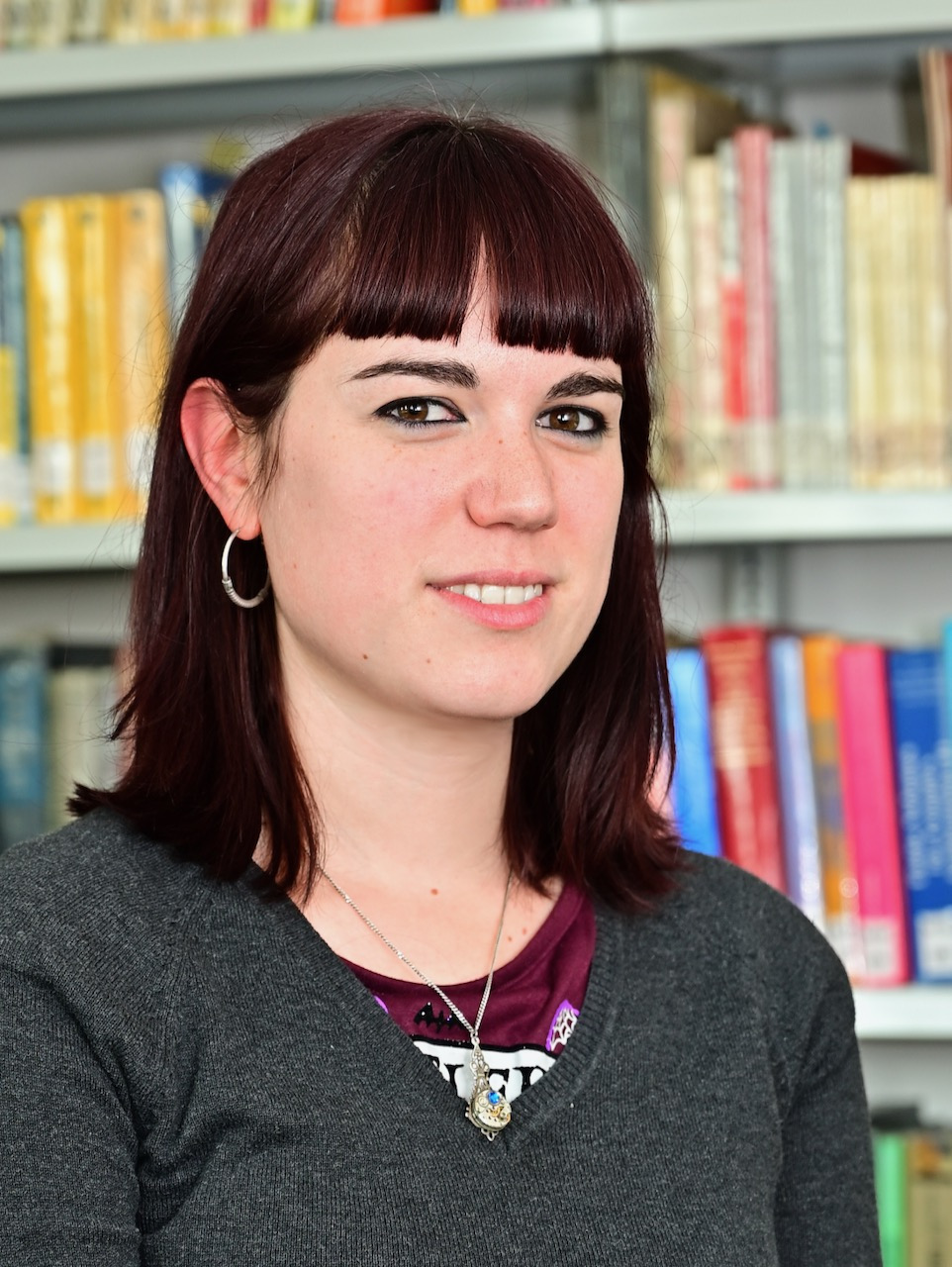}}]{Idoia Cortes Garcia}
    received her Bachelor in mathematics and Master in modelling for science and engineering from the Autonomous University of Barcelona. She obtained her Ph.D.\ in electrical engineering from the Technical University of Darmstadt in 2020. Since 2022 she is an assistant professor at the group of Dynamics and Control from the department of mechanical engineering at Eindhoven University of Technology. Her research interests include coupled multiphysical dynamical systems, differential algebraic equations, efficient time domain (co-)simulation methods and hybrid modelling approaches.
\end{IEEEbiography}
\vspace{-1ex}
\begin{IEEEbiography}[{\includegraphics[width=1in,height=1.25in,clip,keepaspectratio]{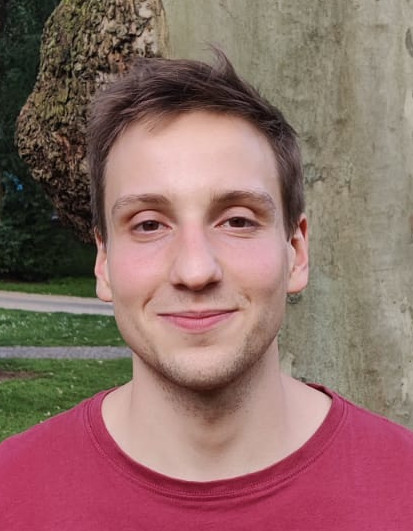}}]{Peter F. Förster}
    received his B.Sc.\ and M.Sc.\ degrees in electrical engineering from Technical University of Darmstadt. Currently, he is a doctoral researcher at the centre for analysis, scientific computing and applications at Eindhoven University of Technology and in the computational electromagnetics group at Technical University of Darmstadt. His research interests include circuit simulation, differential-algebraic equations, machine learning and digital twins.
\end{IEEEbiography}
\vspace{-1ex}
\begin{IEEEbiography}[{\includegraphics[width=1in,height=1.25in,clip,keepaspectratio]{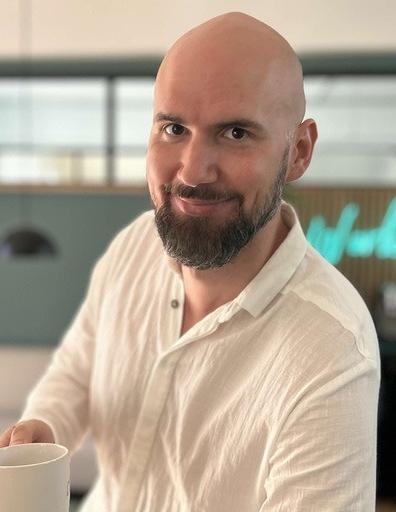}}]{Lennart Jansen}
    received his M.Sc.\ degree in mathematics from Universität zu Köln and his Ph.D.\ degree in mathematics from Universität zu Berlin. Afterwards, he held a post-doc position at Heinrich-Heine-Universität in Düsseldorf before starting work in the industry. One of the main research topics he worked on in the industry was AI-based recommendations for design improvement in the automobile industry. He switched to the photovoltaic industry in 2021 and currently works there on optimal control for combined electric-water-gas networks.
\end{IEEEbiography}
\vspace{-1ex}
\begin{IEEEbiography}[{\includegraphics[width=1in,height=1.25in,clip,keepaspectratio]{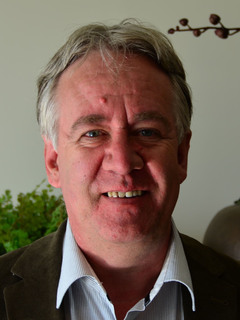}}]{Wil Schilders}
    received the M.Sc.\ degree in mathematics from Radboud University Nijmegen, Nijmegen, The Netherlands, in 1978, and the Ph.D.\ degree from the Trinity College Dublin, Dublin, Ireland, in 1980. He has been working in the electronics industry with Philips Research Laboratories, Eindhoven, The Netherlands, since 1980, and NXP, since 2006, where he developed algorithms for simulating semiconductor devices, electronic circuits, organic light emitting diodes, and electromagnetic problems (TV tubes, interconnects, and magnetic resonance imaging). Since 1999, he has been a Part-Time Professor of numerical mathematics for industry with the Technical University of Eindhoven, Eindhoven. He is currently the Managing Director of the Platform for Mathematics in The Netherlands. He is also the President of the European Consortium for Mathematics in Industry.
\end{IEEEbiography}
\vspace{-1ex}
\begin{IEEEbiography}[{\includegraphics[width=1in,height=1.25in,clip,keepaspectratio]{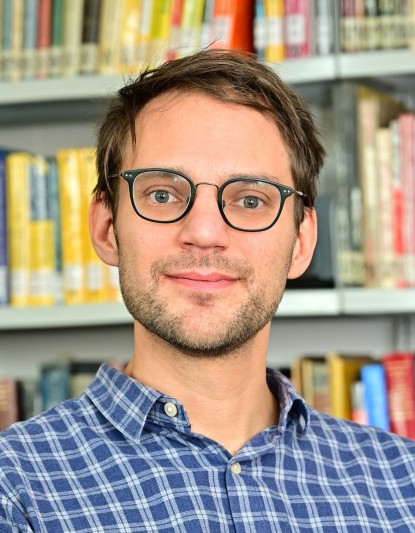}}]{Sebastian Schöps}
    received the M.Sc.\ degree in business mathematics and the joint Ph.D.\ degree from Bergische Universität Wuppertal and Katholieke Universiteit Leuven in mathematics and physics. He was appointed as a Professor of Computational Electromagnetics at Technische Universität Darmstadt within the interdisciplinary center of computational engineering, in 2012. His current research interests include coupled multi-physical problems, bridging computer aided design and simulation, parallel algorithms for high performance computing, digital twins, uncertainty quantification, and machine learning.
\end{IEEEbiography}
\end{document}